\documentclass[journal]{IEEEtran}
\usepackage{amsmath,graphicx,bm,amssymb,amsthm,enumerate,epsfig,psfrag,cases}
\renewcommand{\(}{\left(}
\renewcommand{\)}{\right)}
\renewcommand{\[}{\left[}
\renewcommand{\]}{\right]}

\newcommand{\m}{\mathbf{m}}

\newcommand{\s}{\mathbf{s}}

\renewcommand{\S}{\mathbf{S}}

\newcommand{\y}{\mathbf{y}}
\renewcommand{\H }{\mathbf{H}}
\newcommand{\z}{\mathbf{z}}
\newcommand{\w}{\mathbf{w}}
\newcommand{\p}{\mathbf{p}}
\newcommand{\W}{\mathbf{W}}
\newcommand{\E}{\mathbf{E}}
\newcommand{\R}{\mathbf{R}}
\renewcommand{\r}{\mathbf{r}}

\newcommand{\D}{\mathbf{D}}
\newcommand{\Z}{\mathbf{Z}}

\newcommand{\x}{\mathbf{x}}

\newcommand{\I}{\mathbf{I}}
\newcommand{\J}{\mathbf{J}}

\renewcommand{\P}{\mathbf{P}}

\newcommand{\A}{\mathbf{A}}

\newcommand{\U}{\mathbf{U}}
\newcommand{\M}{\mathbf{M}}

\renewcommand{\u}{\mathbf{u}}

\newcommand{\Q}{\mathbf{Q}}

\newcommand{\q}{\mathbf{q}}

\newcommand{\X}{\mathbf{X}}
\newcommand{\Y}{\mathbf{Y}}

\newcommand{\Tr}[1]{{\rm{Tr}}\left(#1\right)}

\newcommand{\End}[1]{{\rm{End}}}

\renewcommand{\arg}[1]{{\rm{arg}}#1}

\newcommand{\diag}[1]{{\rm{diag}}\left\{#1\right\}}

\renewcommand{\vec}[1]{{\rm{vec}}\(#1\)}
\newcommand{\vech}[1]{{\rm{vech}}\(#1\)}

\newcommand{\mat}[1]{{\rm{mat}}\(#1\)}

\newcommand{\rank}[1]{{\rm{rank}}#1}

\newtheorem{lemma}{Lemma}
\newtheorem{definition}{Definition}
\newtheorem{theorem}{Theorem}

\newtheorem{corollary}{Corollary}

\newtheorem{assm}{Assumption}

\DeclareMathOperator*{\argmin}{\arg\!\min}

\newcommand{\norm}[1]{\left\lVert#1\right\rVert}

\usepackage{fontenc}
\usepackage{inputenc}
\usepackage[square,sort,compress,comma,numbers]{natbib}
\usepackage{epstopdf,pst-node}
\usepackage{mathdots,enumitem}
\usepackage{mathtools}

\begin{document}
\title{Joint Covariance Estimation with Mutual Linear Structure}
\author{Ilya~Soloveychik and Ami~Wiesel, \\
 Rachel and Selim Benin School of Computer Science and Engineering, The Hebrew University of Jerusalem, Israel
\thanks{This work was partially supported by the Intel Collaboration Research Institute for Computational Intelligence and Israel Science Foundation.}
\thanks{The results were partially presented at the $40$-th IEEE International Conference on Acoustics, Speech, and Signal Processing, April, 19-24, 2015, Brisbane, Australia.}
}

\maketitle
\begin{abstract}
We consider the problem of joint estimation of structured covariance matrices. Assuming the structure is unknown, estimation is achieved using heterogeneous training sets. Namely, given groups of measurements coming from centered populations with different covariances, our aim is to determine the mutual structure of these covariance matrices and estimate them. Supposing that the covariances span a low dimensional affine subspace in the space of symmetric matrices, we develop a new efficient algorithm discovering the structure and using it to improve the estimation. Our technique is based on the application of principal component analysis in the matrix space. We also derive an upper performance bound of the proposed algorithm in the Gaussian scenario and compare it with the Cramer-Rao lower bound. Numerical simulations are presented to illustrate the performance benefits of the proposed method.
\end{abstract}


\IEEEpeerreviewmaketitle

\section{Introduction}
Large scale covariance matrix estimation using a small number of measurements is a fundamental problem in modern multivariate statistics. In such environments the Sample Covariance Matrix (SCM) may provide a poor estimate of the true covariance. The scarce amount of samples is especially a problem in financial data analysis where few stationary monthly observations of numerous stock indexes are used to estimate the joint covariance matrix of the stock returns \cite{laloux2000random, ledoit2003improved}, bioinformatics where clustering of genes is obtained based on gene sequences sampled from a small population \cite{schafer2005shrinkage}, computational immunology where correlations among mutations in viral strains are estimated from sampled viral sequences and used as a basis of novel vaccine design \cite{dahirel2011coordinate, quadeer2013statistical}, psychology where the covariance matrix of multiple psychological traits is estimated from data collected on a group of tested individuals \cite{steiger1980tests}, or electrical engineering at large where signal samples extracted from a possibly short time window are used to retrieve parameters of the signal \cite{scharf1991statistical}. 

The most common approach to work around the sample deficiency problem is to introduce prior information, which reduces the number of degrees of freedom in the model and allows accurate estimation with few samples. Prior knowledge on the structure can originate from the physics of the underlying phenomena or from similar datasets, see e.g. \cite{fuhrmann1991application, pollock2002circulant, stoica2011spice}. The most widely used approaches here are shrinkage towards a known matrix \cite{ledoit2004well, chen2011robust}, or assuming the true covariance matrix to be structured \cite{snyder1989use, abramovich2007time, wiesel2013time, dembo1989embedding, cai2013optimal, friedman2008sparse, banerjee2008model, cai2012optimal, bickel2008regularized}. We focus on prior knowledge in the form of structural constraints, which are most commonly affine. Probably the most popular of them is the Toeplitz model \cite{snyder1989use, abramovich2007time, wiesel2013time,kavcic2000matrices, asif2005block,fuhrmann1991application, roberts2000hidden,bickel2008regularized}, closely related to it are circulant matrices \cite{dembo1989embedding, cai2013optimal}. In other settings the number of parameters can be reduced by assuming that the covariance matrix is sparse \cite{friedman2008sparse, banerjee2008model, cai2012optimal}. A popular sparse model is the banded covariance, which is associated with time-varying moving average models \cite{bickel2008regularized,wiesel2013time}. Other examples of structured covariances include factor models \cite{engle1990asset}, permutation invariant models \cite{shah2012group}, and many others.

An important common feature of the works listed in the previous paragraph is that they consider a single and static environment where the structure of the true parameter matrix, or at least the class of structures, as in the sparse case, is known in advance. Often, this is not the case and techniques are needed to learn the structure from the observations. A typical approach is to consider multiple datasets sharing a similar structure but non homogeneous environments \cite{guo2011joint, besson2008covariance, bidon2008bayesian}. This is, for example, the case in covariance estimation for classification across multiple classes \cite{danaher2014joint}. A related problem addresses tracking a time varying covariance throughout a stream of data \cite{wiesel2013time, moulines2005recursive}, where it is assumed that the structure changes at a slower rate than the covariances themselves \cite{ahmed2009recovering}. Here too, it is natural to divide this stream of data into independent blocks of measurements.

Our goal is to first rigorously state the problem of joint covariance estimation with linear structure, and derive lower performance bounds for the family of unbiased estimators. Secondly, we propose and analyze new algorithms of learning and exploiting the linear structure to improve the covariance estimation. More exactly, given a few groups of measurements having different second moments each, our target is to determine the underlying low dimensional affine subspace containing or approximately containing the covariance of all the groups. The discovered hyperplane is further used to improve the matrix parameter estimation. Most of the previous works considered particular cases of this method, e.g. factor models, entry-wise linear structures like in sparse and banded cases, or specific patterns like in Toeplitz, circulant and other models. Our algorithm treats the SCM of the heterogeneous populations as vectors in the space of matrices and is based on application of the principal component analysis (PCA) to learn their low-dimensional structure. To make the performance analysis a tractable problem, we assume the data is Gaussian in the corresponding sections, however noting that the algorithm itself provides a good estimation in a much wider class of distributions.

The paper is organized as following. First, we introduce the notations, state the problems and provide examples of affine structures motivating the work. Then we derive the lower performance bound for a class of unbiased estimators. Next we propose our algorithm and provide its upper performance bound. We conclude by numerical simulations demonstrating the performance advantages of the proposed technique and supporting our theoretical results. The Appendix section contains all the proofs and auxiliary statements.

\subsection{Notations}
Denote by $\mathcal{S}(p)$ the $l = \frac{p(p+1)}{2}$ dimensional linear space of $p \times p$ symmetric real matrices. $\I_d$ stands for the $d \times d$ identity matrix. For a matrix $\M$ its Moore-Penrose generalized inverse is denoted by $\M^\dagger$. For any two matrices $\M$ and $\P$ we denote by $\M\otimes\P$ their tensor (Kronecker) product. $\sigma_1(\M) \geqslant \dots \geqslant \sigma_r(\M) \geqslant 0$ stand for the singular values of a rectangular matrix $\M$ of rank not greater than $r$. If $\M \in \mathcal{S}(p)$, we denote its eigenvalues by $\lambda_1(\M) \geqslant \dots \lambda_p(\M)$. $\norm{\cdot}_F$ denotes the Frobenius, and $\norm{\cdot}_2$ - the spectral norms of matrices. For any symmetric matrix $\S$, $\s = \vech{\S}$ is a vector obtained by stacking the columns of the lower triangular part of $\S$ into a single column. In addition, given an $l$ dimensional column vector $\m$ we denote by $\mat{\m}$ the inverse operator constructing a $p \times p$ symmetric matrix such that $\vech{\mat{\m}}=\m$. Due to this natural linear bijection below we often consider subsets of $\mathcal{S}(p)$ as subsets of $\mathbb{R}^l$. In addition, let $\vec{\S}$ be a $p^2$ dimensional vector obtained by stacking the columns of $\S$, and denote by $\mathcal{I}$ its indices corresponding to the related elements of $\vech{\S}$.

\section*{Acknowledgment}
The authors are grateful to Raj Rao Nadakuditi for his useful remarks and important suggestions, which helped us to significantly improve the paper.

\section{Problem Formulation and Examples}
\label{pr_f}
Consider the heterogeneous measurements model, namely, assume we are given $K\geqslant l = \frac{p(p+1)}{2}$ groups of real $p$ dimensional normal random vectors
\begin{equation}
\x_1^1,\dots,\x_1^n,\;\;\dots,\;\; \x_k^1,\dots,\x_k^n,\;\;\dots,\;\; \x_K^1,\dots,\x_K^n.
\label{msts}
\end{equation}
with $n$ i.i.d. (independent and identically distributed) samples inside each group and covariances
\begin{equation*}
\Q_k = \mathbb{E}[\x_k^1\x_k^{1T}],\;\; k=1,\dots,K.
\end{equation*}
We assume that 
\begin{equation}
\Q_k \in \mathcal{P}(p),\;\; k=1,\dots,K,
\label{q_def}
\end{equation}
belong to an $r$ dimensional affine subspace of $\mathcal{S}(p)$. Our main goal is to estimate this subspace and use it to improve the covariance matrices estimation. As a matter of application we will assume that $r$ is known or not known in advance. In the latter case we will also explain how to determine it from the data.

Let us list the most common affine covariance constraints naturally appearing in typical signal processing applications.
\begin{itemize}
[leftmargin=*]
\item {\bf{Diagonal}}:
The simplest example of a structure is given by diagonal matrices. The covariance matrix is diagonal when the noise variates are independent, or can be assumed so with great precision. In this case the low dimensional subspace containing the structured matrices is $r=p$ dimensional.

\item {\bf{Banded}}: In a similar manner it is often reasonable to assume non-neighboring variates of the sampled vectors to be uncorrelated. Claiming that $i$-th element of the random vector is uncorrelated with the $h$-th if $|i-h|>b$ leads to $b$-banded covariance structure. The subspace of symmetric $b$-banded matrices constitutes an $r=\frac{(2p-b)(b+1)}{2}$ dimensional subspace inside $\mathcal{S}(p)$. Banded covariance matrices often naturally appear in time-varying moving average models or in spatially distributed networks, \cite{bickel2008regularized, abramovich2007time, kavcic2000matrices, asif2005block}.

\item {\bf{Circulant}}:
\label{circ_def}
The next common type of structured covariance matrices are symmetric circulant matrices, defined as
\begin{equation}
\M =
 \begin{pmatrix}
  m_1 & m_2 & m_3 & \dots & m_p \\
  m_p & m_1 & m_2 & \dots & m_{p-1}\\
  \vdots & \vdots & \vdots & \ddots & \vdots \\
  m_2 & m_3 & m_4 & \dots & m_1
  \end{pmatrix},
\label{circ_struct}
\end{equation}
with the natural symmetry conditions such as $m_p = m_2,$ etc. Symmetric circulant matrices belong to an $r=p/2$ dimensional subspace if $p$ is even and $(p+1)/2$ if it is odd. Such matrices are typically used as approximations to Toeplitz matrices \cite{pollock2002circulant} which are associated with signals that obey periodic stochastic properties for example the yearly variation of temperature in a particular location. A special case of such processes are the classical stationary processes, which are ubiquitous in engineering, \cite{dembo1989embedding, cai2013optimal}.

\item {\bf{Toeplitz}}: A natural generalization of circulant are Toeplitz matrices. The covariance matrix appears to possess Toeplitz structure whenever the correlation between the $i$-th and the $h$-th components depend only on the the difference $|i-h|$. The dimension of a subspace of Toeplitz symmetric matrices is $r=p$. The two classical models for spectrum estimation utilizing the Toeplitz structures are the moving average (MA) and the autoregressive (AR) processes, \cite{wiesel2013time,abramovich2007time}. Interestingly, the finite MA($a$) process can be easily shown to be equivalent to $a$-banded Toeplitz covariance model, having the dimension $r=a+1$.

\item {\bf{Proper Complex}}:
Many physical processes can be conveniently described in terms of complex signals. For example, the most frequently appearing model of complex Gaussian noise is the circularly symmetric one \cite{kay1993fundamentals}. Such noise is completely characterized by its mean and rotation invariant hermitian covariance matrix $\Q^C$. Denote centered proper complex distributions as
\begin{equation*}
\x \sim \mathcal{CN}(\bm{0},\Q^C).
\end{equation*}
The real representation of the covariance reads as
\begin{equation}
\Q^R = \frac{1}{2}\begin{pmatrix}
  \operatorname{Re}(\Q^C) & -\operatorname{Im}(\Q^C) \\
  \operatorname{Im}(\Q^C) & \operatorname{Re}(\Q^C)
  \end{pmatrix}.
\label{prop_c_s}
\end{equation}
We see that $\Q^R$ possesses a simple linear $r=p^2/4$ dimensional structure, where $p$ is the dimension of $\Q^R$, which is always even. An analogous reasoning applies to proper quaternion covariances \cite{sloin2014proper}.
\end{itemize}
In the following it will be convenient to use a single matrix notation for the multiple linearly structured matrices. Set
\begin{align}
\q_k &= \vech{\Q_k}, k=1,\dots,K,\\
\Y &= [\q_1,\dots,\q_K],
\end{align}
Using these notation, the prior subspace knowledge is equivalent to a low-rank constraint
\begin{equation}
\Y = \U\Z,
\label{new_param}
\end{equation}
where $\U \in \mathbb{R}^{l \times r}$ and $\Z = [\z_1,\dots,\z_K] \in \mathbb{R}^{r \times K}$. Essentially our problem reduces to estimation of $\Y$ given $K$ groups of i.i.d. measurements $\{\x_k^i\}_{k,i=1}^{K,n}$ coming from the corresponding centered distributions with covariances $\Q_k,k=1,\dots,K$.

\section{Lower Performance Bounds}
\label{mse_b}
Before addressing possible solutions for the above joint covariance estimation problem, it is instructive to examine the inherent performance bounds. In this section we assume the rank $r$ is known and the samples are normally distributed. To obtain such performance bounds we use the Cramer-Rao Bound ($\mathbf{CRB}$) to lower bound the Mean Squared Error ($\mathbf{MSE}$) of any unbiased estimator $\widehat{\Y}$ of $\Y$, defined as
\begin{equation}
\mathbf{MSE} = \mathbb{E}\[\norm{\widehat{\Y} - \Y}_F^2\].
\end{equation}

The $\mathbf{MSE}$ is bounded from below by the trace of the corresponding $\mathbf{CRB}$ matrix. To compute the $\mathbf{CRB}$, for each $i$ we stack the measurements $\x_k^i$ from (\ref{msts}) into a single vector
\begin{equation}
\x^i  = 
\begin{pmatrix}
 \x_1^i \\
 \vdots \\
 \x_K^i
\end{pmatrix}
\sim \mathcal{N}(\bm{0},\Q_{\text{ext}}),\; i=1,\dots,n,
\label{ext_d}
\end{equation}
where the extended covariance is given by
\begin{multline}
\Q_{\text{ext}}(\U,\Z) = \diag{\Q_1,\dots,\Q_K} \\
= \diag{\mat{\U\z_1},\dots,\mat{\U\z_K}}.
\label{param_ex}
\end{multline}
Here the operator $\diag{\Q_1,\dots,\Q_K}$ returns a block-diagonal matrix of size $pK \times pK$ with $\Q_k$-s as its diagonal blocks. The Jacobian matrix of (\ref{ext_d}) parametrized as in (\ref{param_ex}) reads as
\begin{align}
&\J = \frac{\partial\Q_{\text{ext}}}{\partial (\U,\Z)} =
 \begin{pmatrix}
  \frac{\partial \q_1}{\partial \U} & \frac{\partial \q_1}{\partial \z_1} & 0 & \dots & 0 \\
  \frac{\partial \q_2}{\partial \U} & 0 & \frac{\partial \q_2}{\partial \z_2} & \dots & 0 \\
  \vdots & \vdots & \vdots & \ddots & \vdots \\
  \frac{\partial \q_K}{\partial \U} & 0 & 0 & \dots & \frac{\partial \q_K}{\partial \z_K} \\
  \end{pmatrix} \nonumber\\
&=
 \begin{pmatrix}
  \z_1^T \otimes \I_l & \U & 0 & \dots & 0 \\
  \z_2^T \otimes \I_l & 0 & \U & \dots & 0 \\
  \vdots & \vdots & \vdots & \ddots & \vdots \\
  \z_K^T \otimes \I_l & 0 & 0 & \dots & \U \\
  \end{pmatrix}
\in \mathbb{R}^{lK \times (lr+Kr)},
\end{align}
where we have used the following notation:
\begin{equation}
\frac{\partial \q_k}{\partial \U} =
\[\frac{\partial \q_k}{\partial \u_1} \; \frac{\partial \q_k}{\partial \u_2} \; \dots \; \frac{\partial \q_k}{\partial \u_r}\],
\end{equation}
and the formulas
\begin{equation}
\frac{\partial \q_k}{\partial \u_j} = \frac{\partial \U\z_k}{\partial \u_j} = z^j_k\I_l,\quad
\frac{\partial \q_k}{\partial \z_k} = \frac{\partial \U\z_k}{\partial \z_k} = \U.
\end{equation}
\begin{lemma}
\label{rank_lem}
\begin{equation}
\rank(\J) = lr + Kr - r^2
\end{equation}
\end{lemma}
\begin{proof}
The proof can be found in Appendix \ref{ap_1}.
\end{proof}

This lemma implies that $\J$ is rank deficient, as
\begin{equation}
\rank(\J) = lr + Kr - r^2 \leqslant \min[lK, lr+Kr],
\end{equation}
reflecting the fact that the parametrization of $\Q_{\text{ext}}$ or $\Y$ by the pair $(\U,\Z)$ is unidentifiable. Indeed for any invertible matrix $\A$, the pair $(\U\A, \A^{-1}\Z)$ fits as good. Due to this ambiguity the matrix $\mathbf{FIM}(\U,\Z)$ is singular and in order to compute the $\mathbf{CRB}$ we use the Moore-Penrose pseudo-inverse of $\mathbf{FIM}(\U,\Z)$ instead of inverse, as justified by \cite{li2012interpretation}. Given $n$ i.i.d. samples $\x^i,i=1,\dots,n$, we obtain
\begin{equation}
\mathbf{CRB} = \frac{1}{n}\J\mathbf{FIM}(\U,\Z)^\dagger\J^T.
\end{equation}
For the Gaussian population the matrix $\mathbf{FIM}(\U,\Z)$ is given by
\begin{equation}
\mathbf{FIM}(\U,\Z) = \frac{1}{2}\J^T \diag{\[\Q_k^{-1} \otimes \Q_k^{-1}\]_{\mathcal{I},\mathcal{I}}}\J,
\end{equation}
where $\[\M\]_{\mathcal{I},\mathcal{I}}$ is the square submatrix of $\M$ corresponding to the subset of indices from $\mathcal{I}$. The bound on the $\mathbf{MSE}$ is therefore given by
\begin{align}
&\mathbf{MSE} \geqslant \Tr{\mathbf{CRB}} = \frac{1}{n}\Tr{\mathbf{FIM}(\U,\Z)^\dagger \J^T\J} \nonumber\\
&= \frac{2}{n}\Tr{\[\J^T \diag{\[\Q_k^{-1} \otimes \Q_k^{-1}\]_{\mathcal{I},\mathcal{I}}}\J\]^\dagger\J^T\J}. 
\label{crb_e}
\end{align}
Denote
\begin{equation}
\underline\lambda = \min_k\lambda_p(\Q_k),\quad \overline\lambda = \max_k \lambda_1(\Q_k).
\end{equation}
To get more insight on (\ref{crb_e}) we bound it from below
\begin{align}
\mathbf{MSE} &\geqslant \frac{2\underline\lambda^2}{n}\Tr{\[\J^T\J\]^\dagger \J^T\J} \nonumber \\
&= \frac{2\underline\lambda^2}{n}\rank(\J) = \frac{2\underline\lambda^2}{n}(lr + Kr - r^2). 
\label{crb_b}
\end{align}
The dependence on the model parameters here is similar to that obtained by \cite{tang2011lower} for the 
problem of low-rank matrix reconstruction. An important quantity is the marginal $\mathbf{MSE}$ per one matrix $\Q_k$, which is proportional to
\begin{equation}
\frac{\mathbf{MSE}}{K} \sim \frac{lr -r^2}{Kn}+ \frac{r}{n}.
\label{marg_mse}
\end{equation}

\section{TSVD Algorithm}
\label{jce_alg}
\subsection{The Basic Algorithm}
In this section we present our algorithm for recovery of the true underlying covariances $\Q_1,\dots,\Q_K$. We make use of the representation (\ref{new_param}) of $\Y$ with a small change in notation consisting in separation of the columns mean
\begin{equation}
\Y = \U\Z + \u_0\cdot[1,\dots,1].
\end{equation}
This is done in order to improve the performance of the proposed PCA-based algorithm. Consider the SCM of the $k$-th group of measurements
\begin{equation}
\S_k = \frac{1}{n} \sum_{i=1}^n \x_i^k\x_i^{kT},
\label{scm_def}
\end{equation}
and let
\begin{equation}
\s_k = \vech{\S_k}.
\end{equation}
Denote by
\begin{equation}
\S = [\s_1,\dots,\s_K],
\end{equation}
the measurement matrix, and compute the average of the columns
\begin{equation}
\widehat{\u}_0 = \frac{1}{K}\sum_{k=1}^K \s_k.
\end{equation}
Consider the matrix
\begin{equation}
\S' = [\s_1-\widehat{\u}_0,\dots,\s_K-\widehat{\u}_0],
\end{equation}
The SVD of $\S'$ reads as
\begin{equation}
\S' = \widehat{\U}  \begin{pmatrix}
\widehat{\bm\Sigma} & 0 \\
0 & \widehat{\bm\Sigma}_n
  \end{pmatrix} \widehat{\W}^T = [\widehat{\U}_1 \widehat{\U}_2]  \begin{pmatrix}
\widehat{\bm\Sigma} & 0 \\
0 & \widehat{\bm\Sigma}_n
  \end{pmatrix} [\widehat{\W}_1 \widehat{\W}_2]^T,
\label{wt_r}
\end{equation}
where the singular values are sorted in the decreasing order and $\widehat{\bm\Sigma} \in \mathbb{R}^{r\times r}$. We propose to use the matrix
\begin{equation}
\widehat{\Y}' = \widehat{\U}_1\widehat{\bm\Sigma}\widehat{\W}_1^T,
\end{equation}
as an estimator of
\begin{equation}
\Y' = \U\Z = [\U_1 \U_2]  \begin{pmatrix}
\bm\Sigma & 0 \\
0 & 0
  \end{pmatrix} [\W_1 \W_2]^T.
\end{equation}
This approach is based on Eckart-Young theorem and we refer to it as Truncated SVD (TSVD) method \cite{stewart1993early}. Finally, for the estimator of $\Y$ we have
\begin{equation}
\widehat{\Y} = \widehat{\Y}' + \widehat{\u}_0\cdot[1,\dots,1].
\end{equation}

\subsection{How to Choose the Rank?}
In real world settings the true rank $r$ of the structure subspace is rarely known in advance and one needs to estimate it from the data before applying the TSVD technique. It is instructive to think about rank estimation, followed by TSVD, simply as thresholding of the data singular values. A large variety of thresholding techniques exist, e.g. hard thresholding, see \cite{donoho2013optimal} and references therein. Unfortunately, almost none of them can be applied in our scenario due to two main reasons. First, most of them require the noise to be independent of the signal, which is not the case in our setting. Second, most of the thresholding approaches rely on prior knowledge about the power or spectral characteristics of the noise, which are known a priori, measured from the secondary data or can be somehow estimated from the samples. In our problem prior information is unavailable and such estimations can not be performed. Instead, we propose a different approach utilizing the fact that the noisy measurements come from Wishart populations.

Consider the expected signal power of $\S_k$
\begin{equation}
\mathbb{E}\norm{\S_k}_F^2 = \frac{n+1}{n}\norm{\Q_k}_F^2 + \frac{1}{n}(\Tr{\Q_k})^2.
\end{equation}
For $n$ large enough
\begin{equation}
\mathbb{E}\norm{\S_k}_F^2 \approx \norm{\Q_k}_F^2\(1+\frac{c}{\rho(\Q_k)}\),
\label{np}
\end{equation}
where $c = \frac{p}{n}$ and
\begin{equation}
\rho(\Q_k) = \frac{p\norm{\Q_k}_F^2}{(\Tr{\Q_k})^2} = \cos^{-2}\(\angle \I,\Q_k\)
\label{q_n}
\end{equation}
is the spherecity coefficient, \cite{ledoit2002some}, measuring how close is $\Q_k$ to the identity matrix. Now (\ref{np}) implies that the unknown squared norm of the true covariance $\norm{\Q_k}_F^2$ is given by
\begin{equation}
\norm{\Q_k}_F^2 = \mathbb{E}\norm{\S_k}_F^2\(1+\frac{c}{\rho(\Q_k)}\)^{-1}.
\end{equation}
We use the following result to estimate the unknown $\rho(\Q_k)$.
\begin{lemma}\cite{ledoit2002some}
Let $p$ and $n$ both tend to infinity in such a way that $\frac{p}{n} \to c$, then
\begin{equation}
\rho(\S_k) \to \rho(\Q_k) +c.
\end{equation}
\end{lemma}
Based on this lemma, use $\widehat\rho(\Q_k) = \rho(\S_k) - c$ as an estimate of $\rho(\Q_k)$ and $\norm{\S_k}$ as an estimate of $\mathbb{E}\norm{\S_k}_F^2$ to obtain
\begin{equation}
\norm{\Q_k}_F^2 \approx \norm{\S_k}_F^2 \(1 - \frac{c}{\rho(\S_k)}\) = \norm{\S_k}_F^2 - \frac{(\Tr{\S_k})^2}{n}.
\end{equation}
The ratio of the desired signal's power $\norm{\Y}_F^2$ to the power of measurements, $\mathbb{E}\norm{\S}_F^2$ can now be estimated by
\begin{equation}
\alpha(\S_1,\dots,\S_K) = \frac{\norm{\Y}_F^2}{\norm{\S}_F^2} = 1- \frac{1}{n}\frac{\sum_{k=1}^K (\Tr{\S_k})^2}{\sum_{k=1}^K \norm{\S_k}_F^2}.
\end{equation}
This derivation suggests a simple rule of thumb for thresholding the spectrum. As an estimate of the rank $\widehat{r}$ we take the number of largest singular values of $\S$ carrying the fraction $\alpha$ of the signal's energy,
\begin{equation}
\widehat{r} = \argmin_L \Bigg|\sum_{l=1}^L \sigma_l^2(\S) - \alpha\norm{\S}_F^2\Bigg|.
\label{rankc}
\end{equation}
Below we test this method of rank recovery in comparison to different approaches by numerical simulations.

\subsection{TSVD Upper Performance Bound}
\label{sec_tsvd}
In this section we provide the performance analysis of the proposed TSVD algorithm in the Gaussian scenario, assuming $r$ and $\u_0$ are known. We make the following natural assumption.
\begin{assm}
\label{ass}
The condition number of the true matrix parameter $\Y$ is bounded from above by
\begin{equation}
\frac{\sigma_1(\Y)}{\sigma_r(\Y)} \leqslant \kappa \sqrt{K},
\end{equation}
and the smallest singular value of $\Y$ is bounded from below by
\begin{equation}
\sigma_r(\Y) \geqslant \epsilon,
\end{equation}
where $\kappa$ and $\epsilon$ do not depend on $K,p,n$ or $r$.
\end{assm}
In addition, without loss of generality let $\u_0 = 0$.

\begin{theorem}
\label{g_b}
Under Assumption \ref{ass} with probability at least $1-2e^{-cp}$
\begin{equation}
\norm{\widehat{\Y}-\Y}_F^2 \leqslant C_1\overline\lambda^2(\kappa+\overline\lambda)^2\frac{r(K+C_2l)}{n},
\label{hpb}
\end{equation}
where $c, C_1$ and $C_2$ do not depend on $K,p,n$ or $r$.
\end{theorem}
\begin{proof}
Denote by $\R$ the zero mean noise matrix
\begin{equation}
\R = \S - \Y.
\end{equation}
The proof consists of two stages. First, using Weyl's and Wedin's Theorems we establish a deterministic bound on $\norm{\widehat{\Y}-\Y}_F$ depending on $\norm{\R}_2, \sqrt{r}$ and $\kappa$. On the second stage we use concentration of measure results to achieve high-probability bound on the norm $\norm{\R}_2$, yielding the desired result.

\begin{lemma}(Weyl, Theorem 4.11 from \cite{stewart1990matrix})
\label{weyl_thm}
With the notations of Section \ref{jce_alg},
\begin{equation}
\max_{1\leqslant j\leqslant r}|\sigma_j(\widehat{\bm\Sigma}) - \sigma_j(\Y)| \leqslant \norm{\R}_2.
\end{equation}
\end{lemma}

\begin{lemma}(Wedin, Theorem 4.4 from \cite{stewart1990matrix})
With the notations of Section \ref{jce_alg},
\begin{equation}
\sigma_{\min}([\widehat{\U}^\perp]^T\U^\perp), \sigma_{\min}(\widehat{\U}^T\U) \geqslant \sqrt{1-\frac{\norm{\R}_2^2}{\sigma_r^2(\Y)}}.
\end{equation}
\end{lemma}

\begin{lemma}
\label{m_t}
When $r$ and $\u_0$ are known,
\begin{equation}
\norm{\widehat{\Y}-\Y}_F \leqslant \sqrt{r}\norm{\R}_2\(\sqrt{2}\frac{\sigma_1(\Y)+2+\norm{\R}_2}{\sigma_r(\Y)} + 1\).
\label{thm_in}
\end{equation}
\end{lemma}
\begin{proof}
Use the triangle inequality, Weyl's and Wedin's theorems to get
\begin{align}
&\norm{\widehat{\Y}-\Y}_F \leqslant \norm{\widehat{\U}_1\widehat{\bm\Sigma}\widehat{\W}_1 - \U_1\widehat{\bm\Sigma}\widehat{\W}_1}_F \nonumber\\
&+ \norm{\U_1\widehat{\bm\Sigma}\widehat{\W}_1 - \U_1\bm\Sigma\widehat{\W}_1}_F + \norm{\U_1\bm\Sigma\widehat{\W}_1 - \U_1\bm\Sigma\W_1}_F \nonumber\\
&\leqslant \norm{\widehat{\bm\Sigma}}_2\norm{\widehat{\U}_1 - \U_1}_F + \norm{\widehat{\bm\Sigma} - \bm\Sigma}_F + \sigma_1(\Y)\norm{\widehat{\W}_1 - \W_1}_F \nonumber\\
&\leqslant \(\norm{\widehat{\bm\Sigma}}_2+\sigma_1(\Y)\)\sqrt{2r}\frac{\norm{\R}_2}{\sigma_r(\Y)} + \norm{\widehat{\bm\Sigma} - \bm\Sigma}_F \nonumber \\
&\leqslant \sqrt{2r}\frac{1+\sigma_1(\Y)}{\sigma_r(\Y)}\norm{\R}_2 + \norm{\widehat{\bm\Sigma} - \bm\Sigma}_2\(\sqrt{r}+\frac{\sqrt{2r}}{\sigma_r(\Y)}\norm{\R}_2\) \nonumber\\
&\leqslant \sqrt{r}\norm{\R}_2\(\frac{\sqrt{2}}{\sigma_r(\Y)}\(\sigma_1(\Y)+1+ \norm{\R}_2\) + 1\).
\end{align}
\end{proof}
This result quantifies the intuition that the error bound should depend on the intrinsic dimension $r$ of the estimated subspace, rather than on the ambient dimension $l$. 

We now proceed to the second stage of the proof and develop a high-probability bound on the spectral norm of $\R$ in the following
\begin{lemma}
\label{r_lem}
with probability at least $1-2e^{-cp}$,
\begin{equation}
\norm{\R}_2 \leqslant \overline\lambda\(\sqrt{\frac{2}{n}} +  C_3\sqrt{\frac{l}{nK}}\),
\label{lle}
\end{equation}
where $C_3$ does not depend on the model parameters $K,p,n$ or $r$.
\end{lemma}
\begin{proof}
The proof can be found in the Appendix \ref{ap_2}.
\end{proof}

Plug (\ref{lle}) into (\ref{thm_in}) to obtain that with probability at least $1-2e^{-cp}$,
\begin{align}
&\norm{\widehat{\Y}-\Y}_F \leqslant \overline\lambda\sqrt{\frac{r}{n}}\(\sqrt{2} +  C_3\sqrt{\frac{l}{K}}\)\times \nonumber \\ &\(\sqrt{2}\frac{\sigma_1(\Y)}{\sigma_r(\Y)}+\frac{1}{\sigma_r(\Y)}+\frac{\overline\lambda}{\sigma_r(\Y)}\(\sqrt{\frac{2}{n}} +  C_3\sqrt{\frac{l}{nK}}\) + 1\) \nonumber\\
&\leqslant C_4\overline\lambda(\kappa+\overline\lambda)\sqrt{\frac{r}{n}}\(\sqrt{2K} +  C_3\overline\lambda\sqrt{l}\),
\end{align}
where $C_4$ does not depend on the model parameters and the last inequality is due to Assumption \ref{ass}. We finally obtain that with probability at least $1-2e^{-cp}$,
\begin{equation}
\norm{\widehat{\Y}-\Y}_F \leqslant C_1\overline\lambda(\kappa+\overline\lambda)\frac{\sqrt{r}(\sqrt{K}+C_2\overline\lambda\sqrt{l})}{\sqrt{n}},
\end{equation}
and the statement follows.
\end{proof}

The obtained performance bound (\ref{hpb}) suggests that the error is bounded with high probability by a product of a linear combination of $K$ and $l$ with $r$, rather than the ambient dimension $l$ of the space. Compared to the bound on the $\mathbf{MSE}$, (\ref{crb_b}), it shows that the proposed TSVD algorithm exhibits near-optimal dependence on the model parameters $K,p,n,r$.

\section{Numerical Simulations}
\subsection{TSVD Performance in Toeplitz Model}
For our first experiment we took a Toeplitz setting with $p=10, \U_0 = \I_p$. The true covariances were positive definite matrices generated as $\Q_k = \U_0 + \sum_{j=1}^rz^j_k\D_j/\norm{\D_j}_F$, where $\D_j$ has ones on the $j$-th and $-j$-th subdiagonals and zeros otherwise, $r=p-1$ and $z_k^j$ were i.i.d. uniformly distributed over the interval $[-\frac{1}{2},\frac{1}{2}]$. Figure \ref{perf_n} shows the dependence of the $\mathbf{MSE}$ on $n$ when $K=l$. In the unknown $r$ case we took $\widehat\alpha$ power threshold for our TSVD algorithm, as defined in (\ref{rankc}). For comparison we also plot the $\mathbf{MSE}$-s of the SCM and its projection onto the known subspace structure and the true $\mathbf{CRB}$ bound given by (\ref{crb_e}). 

Figure \ref{thr_n} compares different thresholding techniques. As the benchmarks we took the Asymptotically Optimal Hard Thresholding (AOHT), proposed by M. Gavish and D.L. Donoho in \cite{donoho2013optimal} and the so-called ``elbow method''. The AOHT approach consists in hard thresholding the singular values of a noisy low rank matrix $\X$ at a fixed level of
\begin{equation}
\tau(\X) = \omega\(\frac{l}{K}\)\sigma_m(\X),
\end{equation}
independent of the true rank value. Here the function $\omega(\beta)$ is defined in \cite{donoho2013optimal} together with the numerical algorithm of its evaluation, and $\sigma_m(\X)$ is the median of the singular values of $\X$. As the authors demonstrate in \cite{donoho2013optimal}, this approach provides asymptotically optimal hard thresholding under specific assumptions on the noise (in particular, the noise column vectors are assumed i.i.d. white Gaussian). We applied the threshold $\tau(\X)$ to the original measurement $\S$ (AOHT-S in the figure), and its centered counterpart $\S'$ (AOHT-S-c in the figure). As the graph suggests, these two approaches give very close results and perform poorly, compared to our $\alpha$-thresholding, when the number of measurements $n$ in each group is relatively small. The ``elbow method'' consists in thresholding the spectrum of $\X$ at the level of the largest gap between the consecutive singular values of $\X$. This intuitive rule follows the well known observation that the eigenvalues corresponding to the signal and noise usually group into clusters separated by a significant gap. This observation also relies on specific properties of the noise which can be violated in our settings. For comparison, we plot the empirical $\mathbf{MSE}$s of the ``elbow method'' applied to $\S$ (Elbow-S in the figure) and $\S'$ (Elbow-S-c in the figure) correspondingly.

\begin{figure}
\centering
\includegraphics[width = 3.6in]{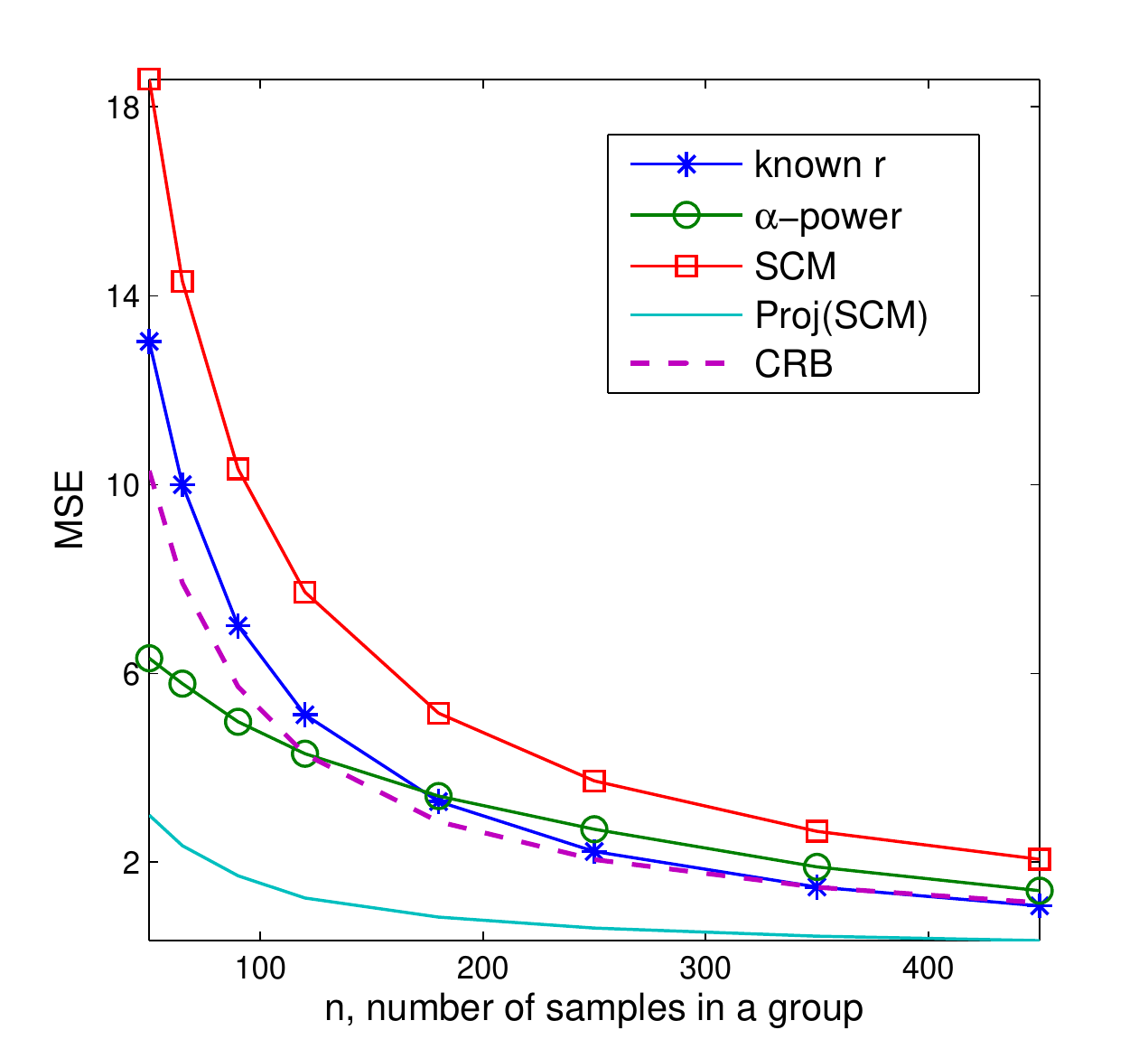}
\caption{TSVD algorithm performance.}
\label{perf_n}
\centering
\includegraphics[width = 3.6in]{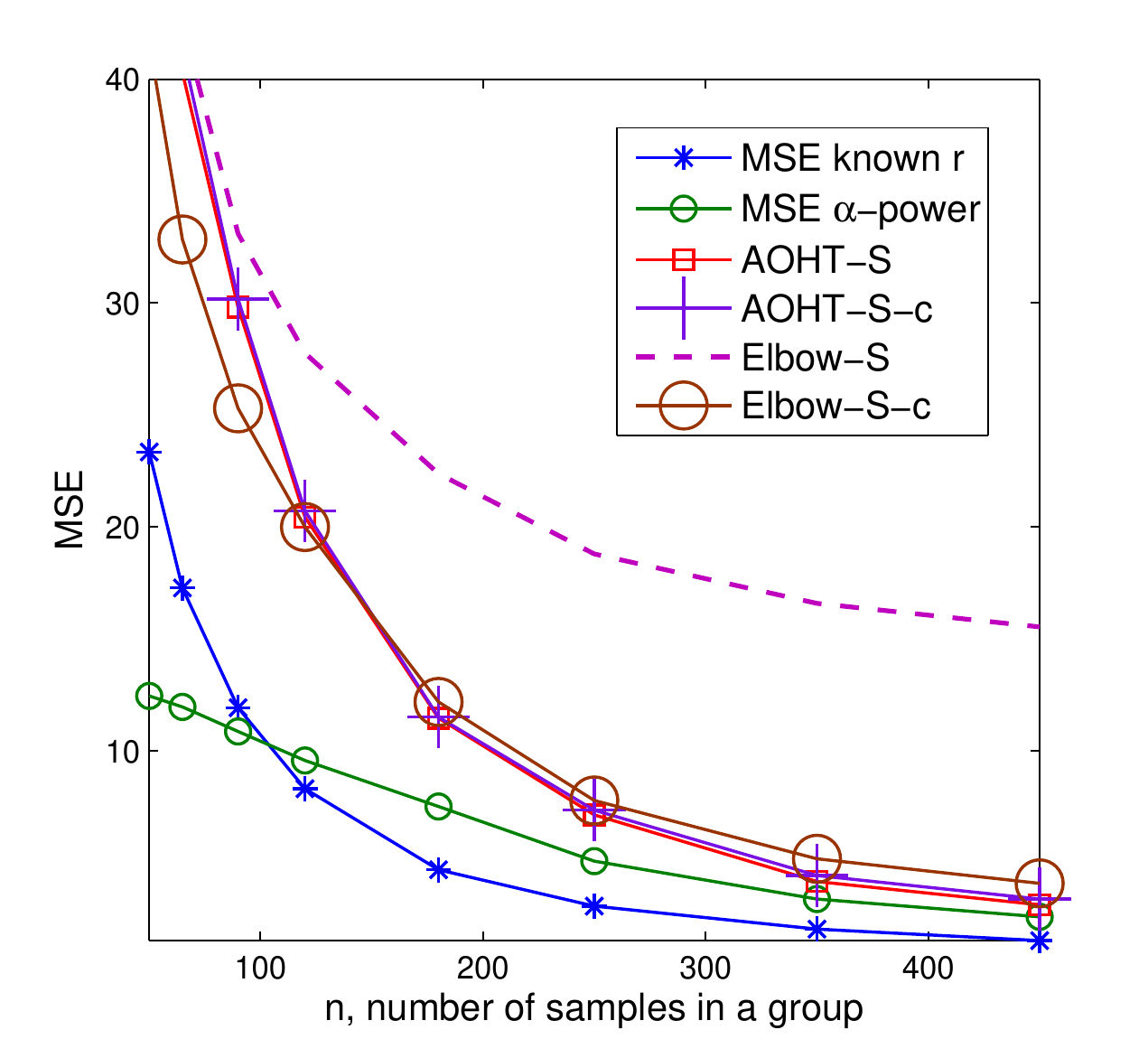}
\caption{TSVD algorithm different thresholding methods.}
\label{thr_n}
\end{figure}

In the third experiment we set $n=100$ fixed and explored the dependence of the $\mathbf{MSE}$ on the number of groups $K$ in the same setting as before. Figure \ref{perf_k} verifies that the marginal $\mathbf{MSE}$ depends on $K$ as predicted by formula (\ref{marg_mse}).
\begin{figure}
\centering
\includegraphics[width = 3.6in]{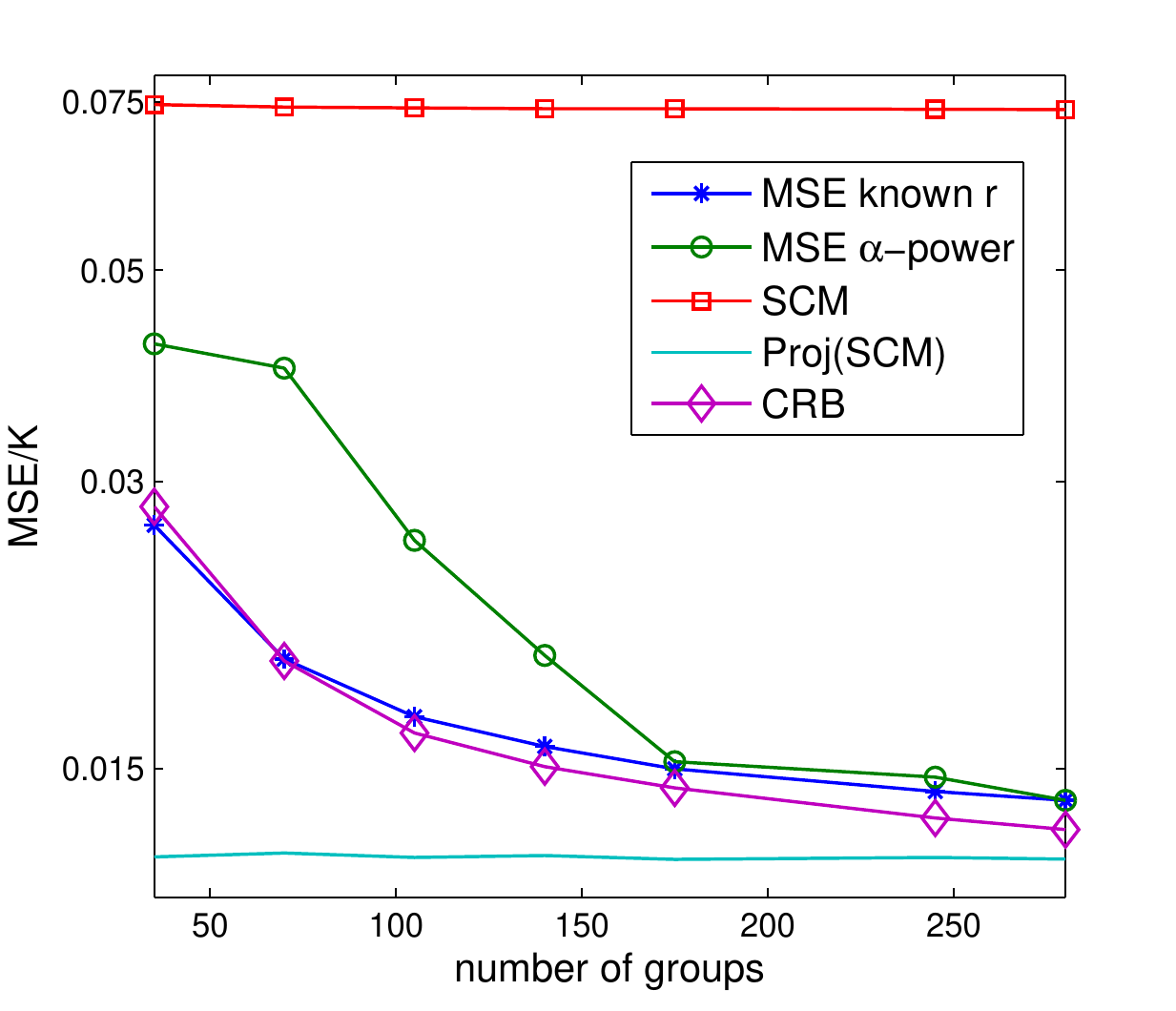}
\caption{Marginal TSVD algorithm performance, $n=100$.}
\label{perf_k}
\end{figure}

\subsection{TSVD Time-tracking}
For the second experiment we considered the problem of tracking a time-varying covariance in complex populations. We used the Data Generating Process (DGP) of Patton and Sheppard, \cite{patton2009evaluating}, which allows for dynamically changing covariances in the spirit of a multivariate GARCH-type model, \cite{bollerslev1988capital, hansson1998testing}. One of variations of this DGP suggests the following data model:
\begin{equation}
\x_t = \H_t^{1/2}\y_t, t=1,\dots Kn,
\end{equation}
where we assumed the generating data to be proper complex, $\y_t \sim \mathcal{CN}(\bm{0},\I)$ and defined the hermitian time-varying covariance $\H_t$ to change according to the law
\begin{align}
&\widehat{\H}_t = (1-\beta)\H_{t-1} + \beta \M_t\M_t^H, \\
&\H_t = \frac{\widehat{\H}_t}{\norm{\widehat{\H}_t}_F}, t=1,\dots Kn.
\end{align}
Here $\M_t$ are random $p\times p$ matrices with i.i.d. standard normally distributed entries, $\H_0$ is arbitrary positive-definite hermitian and $\beta \in [0,1]$.
\begin{figure}
\centering
\includegraphics[width = 3.6in]{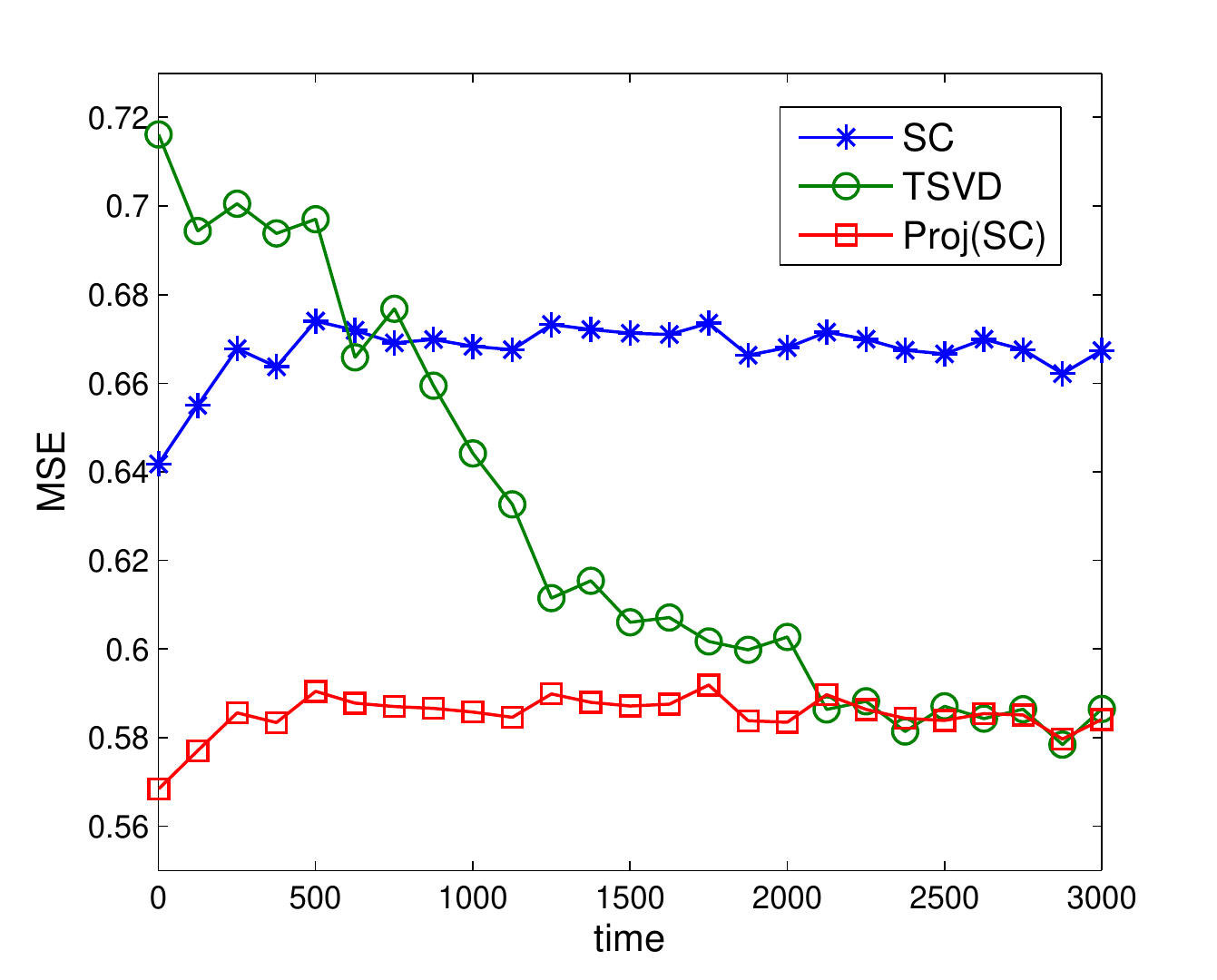}
\caption{TSVD algorithm learning the low-dimensional subspace with time.}
\label{sl30}
\end{figure}

The low-dimensional structure appearing in this setting is due to properness of the covariances (see (\ref{prop_c_s})). In order to explore it, the obtained complex data was represented as double-sized real measurements. Each $n$ clock ticks we formed the SCM $\S_{\frac{t}{n}}$ of the last $n$ measurements, where $t$ was the last time count. Then we concatenated the vector $\vech{\S_{\frac{t}{n}}}$ to the matrix $\Y_{\frac{t}{n}-1} \in \mathbb{R}^{l \times (\frac{t}{n}-1)}$ of growing size to obtain $\Y_{\frac{t}{n}}$ and applied our TSVD algorithm to it. Thus, our structure knowledge was updated every $n$ ticks, and we expected the error of the covariance estimation to decrease with time. We performed the experiment with $p=4, K=100, n=30, \beta=0.01$ and used $90\%$-power threshold to discover the underlying low-dimensional structure. Figure \ref{sl30} compares the temporal behavior of the $\mathbf{MSE}$-s of the SCM, TSVD applied to it and the projection of SCM on the subspace spanned by proper covariances. The $\mathbf{MSE}$-s were obtained by averaging the squared errors over $10000$ iterations.

\section{Conclusion}
In this paper we consider the problem of joint covariance estimation with linear structure, given heterogeneous measurements. The main challenge in this scenario is twofold. At first, the underlying structure is to be discovered, and then utilized to improve the covariance estimation. We propose a PCA-based algorithm in the space of matrices, solving the both tasks simultaneously. In addition, we analyze the performance bounds of this algorithm and show it is close to being optimal. Finally, we demonstrate its advantages through numerical simulations.

\appendices
\section{}
\label{ap_1}
\begin{proof}[Proof of Lemma \ref{rank_lem}]
Let us calculate the column rank of $\J$. Complete the columns of $\U$ to an orthonormal basis $\widehat{\U}$ of $\mathbb{R}^l$ and multiply $\J$ on the left by an invertible square matrix $\I_K \otimes\widehat{\U}$ to obtain
\begin{equation}
\widehat{\J} = \(\I_K \otimes\widehat{\U}\) \J = \begin{pmatrix}
  \z_1^T \otimes \widehat{\U} & \I_{l \times r} & 0 & \dots & 0 \\
  \z_2^T \otimes \widehat{\U} & 0 & \I_{l \times r} & \dots & 0 \\
  \vdots & \vdots & \vdots & \ddots & \vdots \\
  \z_K^T \otimes \widehat{\U} & 0 & 0 & \dots & \I_{l \times r} \\
  \end{pmatrix},
\end{equation}
where $\I_{l \times r}$ denotes the first $r$ columns of $\I_l$. Let us look at the first $l$ rows of $\widehat{\J}$
\begin{equation}
\widehat{\J}_1 = \begin{pmatrix}
  \z_1^T \otimes \widehat{\U} & \I_{l \times r} & 0 & \dots & 0
  \end{pmatrix}.
\end{equation}
By subtracting form the first $l$ columns necessary linear combinations of $\I_{l \times r}$ this matrix can be brought to the form
\begin{equation}
\widehat{\J}_1' = \begin{pmatrix}
  \z_1^T \otimes \widehat{\U}' & \I_{l \times r} & 0 & \dots & 0
  \end{pmatrix},
\end{equation}
where the first $r$ rows of $\widehat{\U}'$ are zero. Performing the same procedure for all the row submatrices $\widehat{\J}_k,\;k=1,\dots,K$ yields
\begin{equation}
\widehat{\J}' = \begin{pmatrix}
  \z_1^T \otimes \widehat{\U}' & \I_{l \times r} & 0 & \dots & 0 \\
  \z_2^T \otimes \widehat{\U}' & 0 & \I_{l \times r} & \dots & 0 \\
  \vdots & \vdots & \vdots & \ddots & \vdots \\
  \z_K^T \otimes \widehat{\U}' & 0 & 0 & \dots & \I_{l \times r} \\
  \end{pmatrix},
\end{equation}
of the same rank as $\widehat{\J}$. The first $l$ columns of $\widehat{\J}'$ read as
\begin{equation}
\widehat{\J}'^1 = \begin{pmatrix}
  z_{11}\widehat{\U}'\\
  z_{21}\widehat{\U}'\\
  \vdots\\
  z_{K1}\widehat{\U}'\\
  \end{pmatrix} = \begin{pmatrix}
  z_{11}\\
  z_{21}\\
  \vdots\\
  z_{K1}\\
  \end{pmatrix} \otimes \widehat{\U}',
\end{equation}
and therefore,
\begin{equation}
\rank(\widehat{\J}'^1) = \rank(\widehat{\U}') = l-r.
\end{equation}
Performing the same operation on the column submatrices $\widehat{\J}'^j,\;j=1,\dots,r$ gives
\begin{equation}
\rank(\J) = \rank(\widehat{\J}') = r(l-r) + Kr = lr + Kr - r^2.
\end{equation}
\end{proof}

\section{}
\label{ap_2}
The following definitions and basic lemmas from the theory of Orlicz spaces can be found in such classical references as \cite{rao1991theory}, or in a more convenient form in recent works such as \cite{van2013bernstein}.
\begin{definition}(Subexponential variable and its Orlicz norm)
Let $x$ be a nonnegative random variable such that
\begin{equation}
\mathbb{P}(x \geqslant \mu + t) \leqslant \left\{ 
\begin{array}{ll}
\exp\(-\frac{t^2}{2\sigma^2}\),\quad t \in \[0,\frac{\sigma^2}{b}\],\\
\exp\(-\frac{t}{2b}\),\quad t \in \(\frac{\sigma^2}{b},\infty\),
\end{array}
\right.
\end{equation}
for some $\mu, \sigma$ and $b$, then we say that $x$ is subexponential with parameters $(\sigma^2, b)$.
The Orlicz norm of $x$ is defined as
\begin{equation}
\norm{x}_{\psi_1} = \inf_{\eta > 0}\left\{\mathbb{E}\exp\(\frac{x}{\eta}\)\leqslant 2\right\}.
\end{equation}
\end{definition}
The intuition behind the subexponential norm basically says that if $\norm{x}_{\psi_1}$ is finite, the $\norm{x}_{\psi_1}$-scaled tails of $x$ are not heavier than that of a standard exponential variable. 

\begin{lemma}
\label{subexp_lem}
A random variable $x$ is subexponential with parameters $(\sigma^2,b)$ iff 
\begin{equation}
\mathbb{E}(\exp(\lambda x)) \leqslant \exp\(\frac{\sigma^2\lambda^2}{2}\),\quad\forall\; |\lambda|<\frac{1}{b}.
\end{equation}
\end{lemma}

\begin{corollary}
\label{orl_nor}
The Orlicz norm of a subexponential variable $x$ with parameters $(\sigma^2,b)$ is bounded by
\begin{equation}
\norm{x}_{\psi_1} \leqslant \max\(\frac{\sigma}{2\ln(2)},\;b\),
\end{equation}
\end{corollary}
\begin{proof}
Lemma \ref{subexp_lem} implies
\begin{equation}
\mathbb{E}\exp\(\frac{x}{\kappa}\) \leqslant \exp\(\frac{\sigma^2}{2\kappa^2}\), \quad \forall \kappa \geqslant b.
\end{equation}
Bound the right hand side of the last inequality by $2$ to get the statement.
\end{proof}

Next we use the following tail bound on a $\chi^2$ variable to calculate the Orlicz norm of its centered version.
\begin{lemma}{\cite{laurent2000adaptive}}
\label{chi_b}
Let $\zeta \sim \chi^2(p)$, then
\begin{equation}
\mathbb{P}\(\zeta - p \geqslant 2\sqrt{pt}+2t\) \leqslant e^{-t},
\end{equation}
\begin{equation}
\mathbb{P}\(p - \zeta \geqslant 2\sqrt{pt}\) \leqslant e^{-t}.
\end{equation}
\end{lemma}

\begin{corollary}
Let $\zeta \sim \chi^2(p)$, then the Orlicz norm of $|\zeta-p|$ is bounded by
\begin{equation}
\norm{|\zeta-p|}_{\psi_1} \leqslant 2\sqrt{p}.
\end{equation}
\end{corollary}
\begin{proof}
Follows by a straight forward calculation from Lemma \ref{chi_b} and Corollary \ref{orl_nor}.
\end{proof}

In the course of proof of Theorem \ref{g_b} we also use the following tail bound on the spectral norm deviations of a standard Wishart matrix from its mean.
\begin{lemma}{\cite{vershynin2010introduction}}
\label{gau_sb}
Let $\W=\frac{1}{n}\sum_{i=1}^n \w_i\w_i^T$, where $\w_i$ are i.i.d. copies of $\w \sim \mathcal{N}(0,\I)$, then
\begin{equation}
\mathbb{P}\(\norm{\W-\I}_2 \geq \sqrt{\frac{p}{n}} + t\) \leqslant 2e^{-nt^2/2},
\end{equation}
\end{lemma}

\begin{theorem}[Theorem 1 from \cite{adamczak2011sharp}]
\label{adam_th}
Let $\r_1,\dots,\r_K \in \mathbb{R}^l$ be independent random vectors (not necessarily identically distributed). Assume that there exists $ \psi$, such that
\begin{equation}
\max_{1\leqslant k \leqslant K} \max_{\p \in \mathcal{S}^{l-1}} \norm{|\langle \r_k, \p \rangle|}_{\psi_1} \leqslant \psi,
\label{cond_1}
\end{equation}
where $\mathcal{S}^{l-1} \subset \mathbb{R}^l$ is a unit sphere, and there exists $\mathcal{K}>1$, such that
\begin{equation}
\mathbb{P}\(\frac{1}{\sqrt{l}}\max_{1\leqslant k \leqslant K} \norm{\r_k} \geqslant \mathcal{K}\max\[1,\(\frac{K}{l}\)^{1/4}\]\) \leqslant e^{-\sqrt{l}},
\label{cond_2}
\end{equation}
then with probability at least $1-2e^{-c\sqrt{l}}$
\begin{equation}
\norm{\R}_2 \leqslant \norm{\E[\R\R^T]}_2^{1/2} + C(\psi+\mathcal{K})\sqrt{\frac{l}{K}},
\end{equation}
where $\R =[\r_1,\dots,\r_K]$ and $c, C$ are universal constants.
\end{theorem}

\begin{proof}[Proof of Lemma \ref{r_lem}]
To make the calculations easier in this proof we assume that the columns of $\R$,
\begin{equation}
\R =[\r_1,\dots,\r_K],
\end{equation}
are constructed as $\r_k = \vec{\S_k-\Q_k}$ and therefore, $\R \in \mathbb{R}^{p^2 \times K}$. This may only affect the numerical constants by the magnitude up to $2$. In order to apply Theorem \ref{adam_th} we need to make sure the conditions (\ref{cond_1}) and (\ref{cond_2}) hold. For this purpose fix a matrix $\P$ and consider the univariate variable $|\langle \r_k, \p \rangle|$, where $\p = \vec{\P}$, for some fixed $k$. Note that
\begin{equation}
\langle \r_k, \p \rangle = \Tr{(\S_k-\Q_k)\P} = \sum_i \x_k^{iT}\P\x_k^i - \Tr{\Q_k\P},
\label{psn}
\end{equation}
and among all the matrices $\P$ of norm one, the Orlicz norm of the right-hand side of (\ref{psn}) is maximized by
\begin{equation}
\P_k = \frac{\Q_k^{-1}}{\norm{\Q_k^{-1}}_F}.
\end{equation}
For this choice of $\P_k$ we obtain
\begin{equation}
|\Tr{(\S_k-\Q_k)\P}| = \frac{|\frac{1}{n}\sum_i\norm{\x_k^i}^2 - p|}{\norm{\Q_k^{-1}}_F} \sim \frac{\frac{1}{n}|\zeta - np|}{\norm{\Q_k^{-1}}_F},
\end{equation}
where $\zeta \sim \chi^2(np)$ and, thus,
\begin{equation}
\max_{\norm{\P}_F=1} \norm{|\Tr{(\S_k-\Q_k)\P}|}_{\psi_1} \leqslant \frac{2\sqrt{np}}{n\norm{\Q_k^{-1}}_F} \leqslant \frac{2\lambda_1(\Q_k)}{\sqrt{n}},
\label{onee}
\end{equation}
where we have used the inequality
\begin{equation}
\norm{\Q_k^{-1}}_F \geqslant \frac{\sqrt{p}}{\lambda_1(\Q_k)}.
\end{equation}
The bound (\ref{onee}) finally yields
\begin{equation}
\max_k \max_{\norm{\P}_F=1} \norm{|\Tr{(\S_k-\Q_k)\P}|}_{\psi_1} \leqslant \frac{2\overline\lambda}{\sqrt{n}},
\end{equation}
and therefore the (\ref{cond_1}) holds true for
\begin{equation}
\psi = \frac{2\overline\lambda}{\sqrt{n}}.
\end{equation}
In order to verify the boundedness condition (\ref{cond_2}), consider the variable
\begin{equation}
\max_k \frac{\norm{\r_k}}{\sqrt{p^2}} = \max_k \frac{\norm{\S_k-\Q_k}_F}{p} \leqslant \frac{\overline\lambda}{\sqrt{p}} \max_k \norm{\W-\I}_2,
\end{equation}
where $\W_k = \Q_k^{-1/2}\S_k\Q_k^{-1/2}$. Denote
\begin{equation}
\delta = \max\[1,\;\(\frac{K}{p^2}\)^{1/4}\]
\end{equation}
and use Lemma \ref{gau_sb} to obtain
\begin{align}
&\mathbb{P}\(\frac{1}{p}\max_k \norm{\r_k} \geqslant \mathcal{K}\delta\) \leqslant \mathbb{P}\(\frac{\overline\lambda}{\sqrt{p}} \max_k \norm{\W_k-\I}_2 \geqslant \mathcal{K} \delta \) \nonumber\\
&\leqslant K\mathbb{P}\(\norm{\W_k-\I}_2 \geqslant \mathcal{K}\delta\frac{\sqrt{p}}{\overline\lambda}\) \nonumber\\
&\leqslant 2K \exp\(-\frac{n}{2}\[\mathcal{K}\delta\frac{\sqrt{p}}{\overline\lambda} - \sqrt{\frac{p}{n}}\]^2\).
\end{align}
Set
\begin{equation}
\mathcal{K} = \frac{\overline\lambda}{\sqrt{n}\delta}(\sqrt{2(1 +\ln{2K}/p)} + 1)
\end{equation}
to get
\begin{equation}
\mathbb{P}\(\frac{1}{p}\max_k \norm{\r_k} \geqslant \mathcal{K}\delta\) \leqslant e^{-p},
\end{equation}
thus the conditions of Theorem 1 from \cite{adamczak2011sharp} are satisfied with the constants $\phi$ and $\mathcal{K}$. Let us bound the spectral norm of $\mathbb{E}[\R\R^T]$. For this purpose note that
\begin{equation}
\norm{\mathbb{E}[\R\R^T]}_2 \leqslant \overline\lambda^2\norm{\mathbb{E}[\M\M^T]}_2,
\end{equation}
where
\begin{equation}
\M = \[\m_1,\dots,\m_K\],
\end{equation}
and $\m_k = \vec{\W_k-\I}$, with $\W_k = \frac{1}{n}\sum_{i=1}^n \w_i^k\w_i^{kT}$, and $\w_i^k \sim \mathcal{N}(0,\I),\;i=1,\dots,n,\;k=1,\dots,K$. Note that for $\zeta \sim \chi^2(n)$,
\begin{equation}
\mathbb{E}[(\zeta-n)^2] = 2n,
\end{equation}
and for $\gamma_i,\;\omega_i \sim \mathcal{N}(0,1),\;i=1,\dots,n$, i.i.d.,
\begin{equation}
\mathbb{E}\[\(\sum_{i=1}^n \gamma_i\omega_i\)^2\] = n,
\end{equation}
to obtain 
\begin{equation}
\mathbb{E}[\M\M^T] = 
\frac{1}{n^2}\begin{pmatrix}
2n\I_p & 0 & 0 \\
0 & n\I_{l-p} & n\I_{l-p} \\
0 & n\I_{l-p} & n\I_{l-p}
\end{pmatrix} \in \mathbb{R}^{p^2\times p^2},
\end{equation}
where for convenience we have ordered the elements of $\vec{\A}$ in such a way that the diagonal of $\A$ goes first, then the upper triangular part and then the lower triangular part. We, therefore, obtain,
\begin{equation}
\norm{\mathbb{E}[\R\R^T]}_2 \leqslant  \frac{2\overline\lambda^2}{n}.
\end{equation}
Finally we obtain that with probability at least $1-2e^{-cp}$,
\begin{align}
&\norm{\R}_2 \leqslant \overline\lambda\sqrt{\frac{2}{n}} + C_2(\phi+\mathcal{K})\sqrt{\frac{p^2}{K}} \nonumber\\ 
&= \overline\lambda\sqrt{\frac{2}{n}} +  \frac{C_2\overline\lambda}{\sqrt{n}}\(2+\frac{\sqrt{2(1 +\ln{2K}/p)} + 1}{\delta}\)\sqrt{\frac{p^2}{K}} \nonumber\\
&\leqslant \overline\lambda\(\sqrt{\frac{2}{n}} +  C_3\sqrt{\frac{p^2}{nK}}\),
\end{align}
where we have used the fact that
\begin{equation}
\frac{\sqrt{2(1 +\ln{2K}/p)} + 1}{\delta} = \frac{\sqrt{2(1 +\ln{2K}/p)} + 1}{\max\[1,\;\(\frac{K}{p^2}\)^{1/4}\]}
\end{equation}
is bounded. Now replace $p^2$ back by $l$, which can at most affect the constants by a factor up to $2$, to get the statement.
\end{proof}

\bibliographystyle{IEEEtran}
\bibliography{ilya_bib}

\end{document}